\documentclass[12pt]{amsart}
\usepackage{txfonts}      
\usepackage{amssymb}
\usepackage{eucal}
\usepackage{amsmath}
\usepackage{amscd}
\usepackage[dvips]{color}
\usepackage{multicol}
\usepackage[all]{xy}           
\usepackage{graphicx}
\usepackage{color}
\usepackage{colordvi}
\usepackage{xspace}
\usepackage[hypertex]{hyperref}

\usepackage[active]{srcltx} 

\begin{document}  

\newcommand{\nc}{\newcommand}
\newcommand{\delete}[1]{}
\nc{\dfootnote}[1]{{}}          
\nc{\ffootnote}[1]{\dfootnote{#1}}
\nc{\mfootnote}[1]{\footnote{#1}} 
\nc{\todo}[1]{\tred{To do:} #1}

\nc{\mlabel}[1]{\label{#1}}  
\nc{\mcite}[1]{\cite{#1}}  
\nc{\mref}[1]{\ref{#1}}  
\nc{\mbibitem}[1]{\bibitem{#1}} 

\delete{
\nc{\mlabel}[1]{\label{#1}  
{\hfill \hspace{1cm}{\bf{{\ }\hfill(#1)}}}}
\nc{\mcite}[1]{\cite{#1}{{\bf{{\ }(#1)}}}}  
\nc{\mref}[1]{\ref{#1}{{\bf{{\ }(#1)}}}}  
\nc{\mbibitem}[1]{\bibitem[\bf #1]{#1}} 
}

\nc{\mkeep}[1]{\marginpar{{\bf #1}}} 

\newtheorem{theorem}{Theorem}[section]
\newtheorem{prop}[theorem]{Proposition}
\newtheorem{defn}[theorem]{Definition}
\newtheorem{lemma}[theorem]{Lemma}
\newtheorem{coro}[theorem]{Corollary}
\newtheorem{prop-def}[theorem]{Proposition-Definition}
\newtheorem{claim}{Claim}[section]
\newtheorem{remark}[theorem]{Remark}
\newtheorem{propprop}{Proposed Proposition}[section]
\newtheorem{conjecture}{Conjecture}
\newtheorem{exam}[theorem]{Example}
\newtheorem{assumption}{Assumption}
\newtheorem{condition}[theorem]{Assumption}
\newtheorem{question}[theorem]{Question}

\renewcommand{\labelenumi}{{\rm(\alph{enumi})}}
\renewcommand{\theenumi}{\alph{enumi}}

\nc{\tred}[1]{\textcolor{red}{#1}}
\nc{\tblue}[1]{\textcolor{blue}{#1}}
\nc{\tgreen}[1]{\textcolor{green}{#1}}
\nc{\tpurple}[1]{\textcolor{purple}{#1}}
\nc{\btred}[1]{\textcolor{red}{\bf #1}}
\nc{\btblue}[1]{\textcolor{blue}{\bf #1}}
\nc{\btgreen}[1]{\textcolor{green}{\bf #1}}
\nc{\btpurple}[1]{\textcolor{purple}{\bf #1}}

\nc{\li}[1]{\textcolor{red}{Li:#1}}
\nc{\cm}[1]{\textcolor{blue}{Chengming: #1}}
\nc{\xiang}[1]{\textcolor{green}{Xiang: #1}}


\nc{\oaset}{\mathbf{O}^{\rm alg}}
\nc{\omset}{\mathbf{O}^{\rm mod}}
\nc{\oamap}{\Phi^{\rm alg}}
\nc{\ommap}{\Phi^{\rm mod}}
\nc{\ioaset}{\mathbf{IO}^{\rm alg}}
\nc{\iomset}{\mathbf{IO}^{\rm mod}}
\nc{\ioamap}{\Psi^{\rm alg}}
\nc{\iommap}{\Psi^{\rm mod}}

\nc{\bwt}{{mass}\xspace}
\nc{\bwts}{{masses}\xspace} \nc{\bop}{{extention}\xspace}
\nc{\ewt}{{mass}\xspace} \nc{\ewts}{{masses}\xspace}
\nc{\tto}{{extended}\xspace} \nc{\Tto}{{Extended}\xspace}
\nc{\tte}{{extended}\xspace} \nc{\gyb}{{generalized}\xspace}
\nc{\Gyb}{{Generalized}\xspace} \nc{\ECYBE}{{ECYBE}\xspace}
\nc{\GAYBE}{{GAYBE}\xspace}

\nc{\adec}{\check{;}}
\nc{\aop}{\alpha}
\nc{\dftimes}{\widetilde{\otimes}} \nc{\dfl}{\succ}
\nc{\dfr}{\prec} \nc{\dfc}{\circ} \nc{\dfb}{\bullet}
\nc{\dft}{\star} \nc{\dfcf}{{\mathbf k}}
\nc{\apr}{\ast}
\nc{\spr}{\cdot}
\nc{\twopr}{\circ}
\nc{\tspr}{\star}
\nc{\sempr}{\ast}
\nc{\disp}[1]{\displaystyle{#1}}
\nc{\bin}[2]{ (_{\stackrel{\scs{#1}}{\scs{#2}}})}  
\nc{\binc}[2]{ \left (\!\! \begin{array}{c} \scs{#1}\\
    \scs{#2} \end{array}\!\! \right )}  
\nc{\bincc}[2]{  \left ( {\scs{#1} \atop
    \vspace{-.5cm}\scs{#2}} \right )}  
\nc{\sarray}[2]{\begin{array}{c}#1 \vspace{.1cm}\\ \hline
    \vspace{-.35cm} \\ #2 \end{array}}
\nc{\bs}{\bar{S}} \nc{\dcup}{\stackrel{\bullet}{\cup}}
\nc{\dbigcup}{\stackrel{\bullet}{\bigcup}} \nc{\etree}{\big |}
\nc{\la}{\longrightarrow} \nc{\fe}{\'{e}} \nc{\rar}{\rightarrow}
\nc{\dar}{\downarrow} \nc{\dap}[1]{\downarrow
\rlap{$\scriptstyle{#1}$}} \nc{\uap}[1]{\uparrow
\rlap{$\scriptstyle{#1}$}} \nc{\defeq}{\stackrel{\rm def}{=}}
\nc{\dis}[1]{\displaystyle{#1}} \nc{\dotcup}{\,
\displaystyle{\bigcup^\bullet}\ } \nc{\sdotcup}{\tiny{
\displaystyle{\bigcup^\bullet}\ }} \nc{\hcm}{\ \hat{,}\ }
\nc{\hcirc}{\hat{\circ}} \nc{\hts}{\hat{\shpr}}
\nc{\lts}{\stackrel{\leftarrow}{\shpr}}
\nc{\rts}{\stackrel{\rightarrow}{\shpr}} \nc{\lleft}{[}
\nc{\lright}{]} \nc{\uni}[1]{\tilde{#1}} \nc{\wor}[1]{\check{#1}}
\nc{\free}[1]{\bar{#1}} \nc{\den}[1]{\check{#1}} \nc{\lrpa}{\wr}
\nc{\curlyl}{\left \{ \begin{array}{c} {} \\ {} \end{array}
    \right .  \!\!\!\!\!\!\!}
\nc{\curlyr}{ \!\!\!\!\!\!\!
    \left . \begin{array}{c} {} \\ {} \end{array}
    \right \} }
\nc{\leaf}{\ell}       
\nc{\longmid}{\left | \begin{array}{c} {} \\ {} \end{array}
    \right . \!\!\!\!\!\!\!}
\nc{\ot}{\otimes} \nc{\sot}{{\scriptstyle{\ot}}}
\nc{\otm}{\overline{\ot}}
\nc{\ora}[1]{\stackrel{#1}{\rar}}
\nc{\ola}[1]{\stackrel{#1}{\la}}
\nc{\pltree}{\calt^\pl}
\nc{\epltree}{\calt^{\pl,\NC}}
\nc{\rbpltree}{\calt^r}
\nc{\scs}[1]{\scriptstyle{#1}} \nc{\mrm}[1]{{\rm #1}}
\nc{\dirlim}{\displaystyle{\lim_{\longrightarrow}}\,}
\nc{\invlim}{\displaystyle{\lim_{\longleftarrow}}\,}
\nc{\mvp}{\vspace{0.5cm}} \nc{\svp}{\vspace{2cm}}
\nc{\vp}{\vspace{8cm}} \nc{\proofbegin}{\noindent{\bf Proof: }}
\nc{\proofend}{$\blacksquare$ \vspace{0.5cm}}
\nc{\freerbpl}{{F^{\mathrm RBPL}}}
\nc{\sha}{{\mbox{\cyr X}}}  
\nc{\ncsha}{{\mbox{\cyr X}^{\mathrm NC}}} \nc{\ncshao}{{\mbox{\cyr
X}^{\mathrm NC,\,0}}}
\nc{\shpr}{\diamond}    
\nc{\shprm}{\overline{\diamond}}    
\nc{\shpro}{\diamond^0}    
\nc{\shprr}{\diamond^r}     
\nc{\shpra}{\overline{\diamond}^r}
\nc{\shpru}{\check{\diamond}} \nc{\catpr}{\diamond_l}
\nc{\rcatpr}{\diamond_r} \nc{\lapr}{\diamond_a}
\nc{\sqcupm}{\ot}
\nc{\lepr}{\diamond_e} \nc{\vep}{\varepsilon} \nc{\labs}{\mid\!}
\nc{\rabs}{\!\mid} \nc{\hsha}{\widehat{\sha}}
\nc{\lsha}{\stackrel{\leftarrow}{\sha}}
\nc{\rsha}{\stackrel{\rightarrow}{\sha}} \nc{\lc}{\lfloor}
\nc{\rc}{\rfloor}
\nc{\tpr}{\sqcup}
\nc{\nctpr}{\vee}
\nc{\plpr}{\star}
\nc{\rbplpr}{\bar{\plpr}}
\nc{\sqmon}[1]{\langle #1\rangle}
\nc{\forest}{\calf} \nc{\ass}[1]{\alpha({#1})}
\nc{\altx}{\Lambda_X} \nc{\vecT}{\vec{T}} \nc{\onetree}{\bullet}
\nc{\Ao}{\check{A}}
\nc{\seta}{\underline{\Ao}}
\nc{\deltaa}{\overline{\delta}}
\nc{\trho}{\tilde{\rho}}

\nc{\rpr}{\circ}
\nc{\dpr}{{\tiny\diamond}}
\nc{\rprpm}{{\rpr}}

\nc{\mmbox}[1]{\mbox{\ #1\ }} \nc{\ann}{\mrm{ann}}
\nc{\Aut}{\mrm{Aut}} \nc{\can}{\mrm{can}}
\nc{\twoalg}{{two-sided algebra}\xspace}
\nc{\colim}{\mrm{colim}}
\nc{\Cont}{\mrm{Cont}} \nc{\rchar}{\mrm{char}}
\nc{\cok}{\mrm{coker}} \nc{\dtf}{{R-{\rm tf}}} \nc{\dtor}{{R-{\rm
tor}}}
\renewcommand{\det}{\mrm{det}}
\nc{\depth}{{\mrm d}}
\nc{\Div}{{\mrm Div}} \nc{\End}{\mrm{End}} \nc{\Ext}{\mrm{Ext}}
\nc{\Fil}{\mrm{Fil}} \nc{\Frob}{\mrm{Frob}} \nc{\Gal}{\mrm{Gal}}
\nc{\GL}{\mrm{GL}} \nc{\Hom}{\mrm{Hom}} \nc{\hsr}{\mrm{H}}
\nc{\hpol}{\mrm{HP}} \nc{\id}{\mrm{id}} \nc{\im}{\mrm{im}}
\nc{\incl}{\mrm{incl}} \nc{\length}{\mrm{length}}
\nc{\LR}{\mrm{LR}} \nc{\mchar}{\rm char} \nc{\NC}{\mrm{NC}}
\nc{\mpart}{\mrm{part}} \nc{\pl}{\mrm{PL}}
\nc{\ql}{{\QQ_\ell}} \nc{\qp}{{\QQ_p}}
\nc{\rank}{\mrm{rank}} \nc{\rba}{\rm{RBA }} \nc{\rbas}{\rm{RBAs }}
\nc{\rbpl}{\mrm{RBPL}}
\nc{\rbw}{\rm{RBW }} \nc{\rbws}{\rm{RBWs }} \nc{\rcot}{\mrm{cot}}
\nc{\rest}{\rm{controlled}\xspace}
\nc{\rdef}{\mrm{def}} \nc{\rdiv}{{\rm div}} \nc{\rtf}{{\rm tf}}
\nc{\rtor}{{\rm tor}} \nc{\res}{\mrm{res}} \nc{\SL}{\mrm{SL}}
\nc{\Spec}{\mrm{Spec}} \nc{\tor}{\mrm{tor}} \nc{\Tr}{\mrm{Tr}}
\nc{\mtr}{\mrm{sk}}

\nc{\ab}{\mathbf{Ab}} \nc{\Alg}{\mathbf{Alg}}
\nc{\Algo}{\mathbf{Alg}^0} \nc{\Bax}{\mathbf{Bax}}
\nc{\Baxo}{\mathbf{Bax}^0} \nc{\RB}{\mathbf{RB}}
\nc{\RBo}{\mathbf{RB}^0} \nc{\BRB}{\mathbf{RB}}
\nc{\Dend}{\mathbf{DD}} \nc{\bfk}{{\bf k}} \nc{\bfone}{{\bf 1}}
\nc{\base}[1]{{a_{#1}}} \nc{\detail}{\marginpar{\bf More detail}
    \noindent{\bf Need more detail!}
    \svp}
\nc{\Diff}{\mathbf{Diff}} \nc{\gap}{\marginpar{\bf
Incomplete}\noindent{\bf Incomplete!!}
    \svp}
\nc{\FMod}{\mathbf{FMod}} \nc{\mset}{\mathbf{MSet}}
\nc{\rb}{\mathrm{RB}} \nc{\Int}{\mathbf{Int}}
\nc{\Mon}{\mathbf{Mon}}
\nc{\remarks}{\noindent{\bf Remarks: }}
\nc{\OS}{\mathbf{OS}} 
\nc{\Rep}{\mathbf{Rep}}
\nc{\Rings}{\mathbf{Rings}} \nc{\Sets}{\mathbf{Sets}}
\nc{\DT}{\mathbf{DT}}

\nc{\BA}{{\mathbb A}} \nc{\CC}{{\mathbb C}} \nc{\DD}{{\mathbb D}}
\nc{\EE}{{\mathbb E}} \nc{\FF}{{\mathbb F}} \nc{\GG}{{\mathbb G}}
\nc{\HH}{{\mathbb H}} \nc{\LL}{{\mathbb L}} \nc{\NN}{{\mathbb N}}
\nc{\QQ}{{\mathbb Q}} \nc{\RR}{{\mathbb R}} \nc{\TT}{{\mathbb T}}
\nc{\VV}{{\mathbb V}} \nc{\ZZ}{{\mathbb Z}}


\nc{\calao}{{\mathcal A}} \nc{\cala}{{\mathcal A}}
\nc{\calc}{{\mathcal C}} \nc{\cald}{{\mathcal D}}
\nc{\cale}{{\mathcal E}} \nc{\calf}{{\mathcal F}}
\nc{\calfr}{{{\mathcal F}^{\,r}}} \nc{\calfo}{{\mathcal F}^0}
\nc{\calfro}{{\mathcal F}^{\,r,0}} \nc{\oF}{\overline{F}}
\nc{\calg}{{\mathcal G}} \nc{\calh}{{\mathcal H}}
\nc{\cali}{{\mathcal I}} \nc{\calj}{{\mathcal J}}
\nc{\call}{{\mathcal L}} \nc{\calm}{{\mathcal M}}
\nc{\caln}{{\mathcal N}} \nc{\calo}{{\mathcal O}}
\nc{\calp}{{\mathcal P}} \nc{\calr}{{\mathcal R}}
\nc{\calt}{{\mathcal T}} \nc{\caltr}{{\mathcal T}^{\,r}}
\nc{\calu}{{\mathcal U}} \nc{\calv}{{\mathcal V}}
\nc{\calw}{{\mathcal W}} \nc{\calx}{{\mathcal X}}
\nc{\CA}{\mathcal{A}}

\nc{\fraka}{{\mathfrak a}} \nc{\frakB}{{\mathfrak B}}
\nc{\frakb}{{\mathfrak b}} \nc{\frakd}{{\mathfrak d}}
\nc{\oD}{\overline{D}}
\nc{\frakF}{{\mathfrak F}} \nc{\frakg}{{\mathfrak g}}
\nc{\frakm}{{\mathfrak m}} \nc{\frakM}{{\mathfrak M}}
\nc{\frakMo}{{\mathfrak M}^0} \nc{\frakp}{{\mathfrak p}}
\nc{\frakS}{{\mathfrak S}} \nc{\frakSo}{{\mathfrak S}^0}
\nc{\fraks}{{\mathfrak s}} \nc{\os}{\overline{\fraks}}
\nc{\frakT}{{\mathfrak T}}
\nc{\oT}{\overline{T}}
\nc{\frakX}{{\mathfrak X}} \nc{\frakXo}{{\mathfrak X}^0}
\nc{\frakx}{{\mathbf x}}
\nc{\frakTx}{\frakT}      
\nc{\frakTa}{\frakT^a}        
\nc{\frakTxo}{\frakTx^0}   
\nc{\caltao}{\calt^{a,0}}   
\nc{\ox}{\overline{\frakx}} \nc{\fraky}{{\mathfrak y}}
\nc{\frakz}{{\mathfrak z}} \nc{\oX}{\overline{X}}

\font\cyr=wncyr10

\nc{\redtext}[1]{\textcolor{red}{#1}}


\title[$\calo$-operators]{$\calo$-operators on associative algebras and dendriform algebras}

\author{Chengming Bai}
\address{Chern Institute of Mathematics\& LPMC, Nankai University, Tianjin 300071, China}
         \email{baicm@nankai.edu.cn}
\author{Li Guo}
\address{Department of Mathematics, Zhejiang University, Hangzhou, Zhejiang, China, and Department of Mathematics and Computer Science,
         Rutgers University,
         Newark, NJ 07102}
\email{liguo@newark.rutgers.edu}

\author{Xiang Ni}
\address{Chern Institute of Mathematics \& LPMC, Nankai
University, Tianjin 300071, P.R.
China}\email{xiangn$_-$math@yahoo.cn}

\date{\today}


\begin{abstract}
An O-operator is a relative version of a Rota-Baxter operator and, in the Lie algebra context, is originated from the operator form of the classical Yang-Baxter equation. We generalize the well-known construction of dendriform dialgebras and trialgebras from Rota-Baxter algebras to a construction from $\calo$-operators. We then show that this construction from $\calo$-operators gives all dendriform dialgebras and trialgebras.
Furthermore there are bijections between certain equivalence classes of invertible $\calo$-operators and certain equivalence classes of dendriform dialgebras and trialgebras.
\end{abstract}


\maketitle


\setcounter{section}{0}
{\ }
\vspace{-1cm}

\section{Introduction}
\mlabel{sec:int}
This paper shows that there is a close tie between two seemingly unrelated objects, namely $\calo$-operators and dendriform dialgebras and trialgebras, generalizing and strengthening a previously established connection from Rota-Baxter algebras to dendriform algebras~\mcite{Ag1,Ag2,E1}.

To fix notations, we let $\bfk$ denote a commutative unitary ring in this paper.
By a $\bfk$-algebra we mean an associative (not
necessarily unitary) $\bfk$-algebra, unless otherwise stated.

\begin{defn}{\rm Let $R$ be a $\bfk$-algebra and let $\lambda\in \bfk$ be given. If a $\bfk$-linear map $P:R\rightarrow R$ satisfies the {\bf Rota-Baxter relation}:
\begin{equation}
P(x)P(y)=P(P(x)y)+P(xP(y))+ \lambda P(xy),
\quad \forall x,y \in R,
\mlabel{eq:rbo}
\end{equation}
then $P$ is called a {\bf Rota-Baxter
operator of weight $\lambda$} and $(R,P)$ is called a {\bf Rota-Baxter
algebra of weight $\lambda$.}
}
\mlabel{de:rb}
\end{defn}
Rota-Baxter
algebras arose from studies in probability and combinatorics in the 1960s~\mcite{Bax,Ca,R1} and have experienced a quite remarkable renaissance in recent years with broad applications in mathematics and physics~\mcite{Ag1,Bai1,CK2,E1,EG1,EGK,Gun,GK1,GK3,GZ}.

On the other hand, with motivation from periodicity of algebraic $K$-theory and operads, dendriform dialgebras were introduced by Loday~\cite{Lo} in the 1990s.
\begin{defn}{\rm
A {\bf dendriform dialgebra} is a triple $(R,\prec,\succ)$ consisting of a $\bfk$-module $R$ and two bilinear operations $\prec$ and $\succ $ on $R$ such that
\begin{equation}
(x\prec y)\prec z=x\prec (y\star z),\;\;(x\succ y)\prec z=x\succ (y\prec
z),\;\;x\succ (y\succ z)=(x\star y)\succ z,
\mlabel{eq:dend}
\end{equation}
for all $x,y,z\in R$. Here
$x\star y=x\prec y+x\succ y$.
}
\mlabel{de:dend}
\end{defn}

Aguiar~\mcite{Ag1} first established the following connection from Rota-Baxter algebras to dendriform dialgebras.
\begin{theorem} {\rm $($\mcite{Ag1,Ag2}$)$}
For a Rota-Baxter $\bfk$-algebra $(R,P)$ of weight zero, the binary operations
\begin{equation}
x\prec_P y=x P(y),\;\;
x\succ_P y=P(x) y,\;\;
\forall x,y\in R,
\mlabel{eq:rbd}
\end{equation}
define a dendriform dialgebra $(R,\prec_P,\succ_P)$.
\mlabel{thm:ag}
\end{theorem}
This defines a functor from the category of Rota-Baxter algebras of weight 0 to the category of dendriform dialgebras. This work has inspired quite a few subsequent studies~\mcite{AL,Bai1,Bai2,BCD,E1,EG1,Guo} that generalized and further clarified the relationship between Rota-Baxter algebras and dendriform dialgebras and trialgebras of Loday and Ronco~\mcite{LR}, including the adjoint functor of the above functor, the related Poincare-Birkhoff-Witt theorem and Gr\"{o}bner-Shirshov basis.

These studies further suggested that there should be a close relationship between Rota-Baxter algebras and dendriform dialgebras. Then it is natural to ask whether every dendriform dialgebra and trialgebra could be derived from a Rota-Baxter algebra by a construction like Eq.~(\mref{eq:rbd}). As later examples show, this is quite far from being true.

Our main purpose of this paper is to show that there is a generalization of the concept of a Rota-Baxter operator that could derive all the dendriform dialgebras and trialgebras. It is given by the concept of an $\calo$-operator on a $\bfk$-module and or a $\bfk$-algebra. Such a concept was first introduced in the context of Lie algebras ~\mcite{Bai1,Bo,Ku} to study the classical Yang-Baxter equations and integrable systems, and was recently generalized and applied to the study of Lax pairs and PostLie algebras~\mcite{BGN2}. In the associative algebra context, $\calo$-operators have been applied to study associative analogues of the classical Yang-Baxter equation~\mcite{BGN}.

For simplicity, we only define $\calo$-operators on modules in the introduction, referring the reader to later sections for the more case of $\calo$-operators on algebras.

Let $(A,\apr)$ be a $\bfk$-algebra. Let $(V,\ell,r)$ be an $A$-bimodule, consisting of a
compatible pair of a left $A$-module $(V,\ell)$ given by $\ell: A\to \End(V)$ and a right
 $A$-module $(V,r)$ given by $r:A\to \End(V)$. A linear map $\aop:V\rightarrow A$ is called an {\bf $\calo$-operator on the module} $V$ if
\begin{equation}
\aop(u)\apr
\aop(v)=\aop(\ell(\aop(u))v)+\aop(ur(\aop(v))),\;\;\forall u,v\in V.
\mlabel{eq:aop0}
\end{equation}
When $V$ is taken to be the $A$-bimodule $(A,L,R)$ associated to the algebra $A$, an $\calo$-operator
on the module is just a Rota-Baxter operator of weight zero.

For an $\calo$-operator $\alpha:V\to A$, define
\begin{equation}
 x \prec_\alpha y = x r(\alpha(y)), \quad x \succ_\alpha y= \ell(\alpha(x))y, \quad \forall x,y\in V.
 \end{equation}
Then as in the case of Rota-Baxter operators, we obtain a dendriform dialgebra $(V,\prec_\alpha,\succ_\alpha)$. We also define an {\bf $\calo$-operator on an algebra} that generalizes a Rota-Baxter operator with a non-zero weight and show that an $\calo$-operator on an algebra gives a dendriform trialgebra. We prove in Section~\mref{ss:od1} that every dendriform dialgebras and trialgebra can be recovered from an $\calo$-operator in this way, in contrary to the case of a Rota-Baxter operator.

In Section~\mref{sec:oda} we further show that the dendriform dialgebra or trialgebra structure on $V$ from an $\calo$-operator $\alpha:V\to A$ transports to a dendriform dialgebra or trialgebra structure on $A$ through $\alpha$ under a natural condition. To distinguish the two dendriform dialgebras and trialgebras from an $\calo$-operator $\alpha:V\to A$, we call them the {\bf dendriform dialgebras and trialgebras on the domain} and the {\bf dendriform dialgebras and trialgebras on the range} of $\alpha$ respectively.

By considering the multiplication on the range $A$, we show that, the correspondence from $\calo$-operators to dendriform dialgebras and trialgebras on the domain $V$ implies a more refined correspondence from $\calo$-operators to dendriform dialgebras and trialgebra on the range $A$ that are compatible with $A$ in the sense that the dialgebra and trialgebra multiplications give a splitting (or decomposition) of the associative product of $A$.
We finally quantify this refined correspondence by providing bijections
between certain equivalent classes of $\calo$-operators with range in $A$ and equivalent classes of compatible dendriform dialgebra and trialgebra structures on $A$.

\smallskip

\noindent {\bf Acknowledgements: }
C. Bai thanks the support by the National Natural Science Foundation of China (10621101), NKBRPC (2006CB805905) and SRFDP
(200800550015). L. Guo thanks the NSF grant DMS 0505445
for support, and thanks the Chern Institute of Mathematics at Nankai University and the Center of Mathematics at Zhejiang University for their hospitalities.

\section{$\calo$-operators and dendriform algebras on the domains}
\mlabel{sec:adend} In this section we study the relationship between
$\calo$-operators and dendriform dialgebras and trialgebras on the domains of these operators.
The related concepts and notations are introduced in
Section~\mref{ss:oop} and \mref{ss:rbd}. Then we show that $\calo$-operators recover all dendriform dialgebras and trialgebras on the domains of the operators.

\subsection{$A$-bimodule $\bfk$-algebras and $\calo$-operators}
\mlabel{ss:oop}

We start with a generalization of the well-known concept of bimodules.

\begin{defn}{\rm
Let $(A,\apr)$ be a $\bfk$-algebra with multiplication $\apr$.
\begin{enumerate}
\item Let $(R,\rpr)$ be a $\bfk$-algebra with multiplication $\rpr$. Let
$\ell, r:A\rightarrow \End_\bfk(R)$ be two linear maps. We call $(R,\rpr,\ell,r)$ or simply $R$ an
{\bf $A$-bimodule $\bfk$-algebra} if $(R,\ell,r)$ is an $A$-bimodule
that is compatible with the multiplication $\rpr$ on $R$. More
precisely, we have
\begin{eqnarray}
\ell(x\apr y)v&=&\ell(x)(\ell(y)v),\;\ell(x)(v\rpr w)=(\ell(x)v)\rpr
w,
\mlabel{eq:twoalg1}\\
vr(x\apr y)&=&(vr(x))r(y),\; (v\rpr w)r(x)=v\rpr (wr(x)),\;
\mlabel{eq:twoalg2}\\
(\ell(x)v)r(y)&=&\ell(x)(vr(y)),\; (vr(x))\rpr w=v\rpr (\ell(x)w),\;
\forall\; x,y\in A, v,w\in R. \mlabel{eq:twoalg3}
\end{eqnarray}
\item
A homomorphism between two $A$-bimodule $\bfk$-algebras $(R_1,\rpr_1,\ell_1,r_1)$ and
$(R_2,\rpr_2,\ell_2,r_2)$ is a $\bfk$-linear map $g: R_1\to R_2$ that is both an $A$-bimodule homomorphism and a $\bfk$-algebra homomorphism.
\end{enumerate}
} \mlabel{de:bimal}
\end{defn}

An $A$-bimodule $(V,\ell,r)$ becomes an $A$-bimodule $\bfk$-algebra if we equip $V$ with the zero multiplication.

For a $k$-algebra $(A,\apr)$ and $x\in A$, define the left and right
actions
$$L(x): A\to A,\ L(x)y=x\apr y\,; \quad R(x): A\to A,\  yR(x)=y\apr x, \quad \forall y\in A.$$
Further define
\begin{equation}
L=L_A:A\rightarrow \End_\bfk(A),\ x\mapsto L(x); \quad
R=R_A:A\rightarrow \End_\bfk(A),\ x\mapsto R(x), \quad x\in A.
\mlabel{eq:can}
\end{equation}
As is well-known, $(A,L,R)$ is an $A$-bimodule. Moreover, $(A,\apr, L,R)$ is an
$A$-bimodule $\bfk$-algebra.
Note that an $A$-bimodule $\bfk$-algebra needs not be a left or
right $A$-algebra. For example, the $A$-bimodule $\bfk$-algebra
$(A,\apr,L,R)$ is an $A$-algebra if and only if $A$ is a commutative
$\bfk$-algebra.

We can now define our generalization~\mcite{BGN} of Rota-Baxter operators.
\begin{defn}{\rm
Let $(A,\apr)$ be a $\bfk$-algebra.
\begin{enumerate}
\item
Let $V$ be an $A$-bimodule. A linear map
$\aop:V\rightarrow A$ is called an {\bf $\calo$-operator on the module} $V$ if $\alpha$
satisfies
\begin{equation}
\aop(u)\apr
\aop(v)=\aop(\ell(\aop(u)v))+\aop(ur(\aop(v))),\;\;\forall u,v\in V. \mlabel{eq:aopv}
\end{equation}
\item
Let $(R,\rpr, \ell,r)$ be an $A$-bimodule $\bfk$-algebra and $\lambda\in \bfk$. A linear map
$\aop:R\rightarrow A$ is called an {\bf $\calo$-operator on the algebra $R$ of weight
$\lambda$} if $\alpha$ satisfies
\begin{equation}
\aop(u)\apr
\aop(v)=\aop(\ell(\aop(u)v))+\aop(ur(\aop(v)))+\lambda\aop(u\rpr
v),\;\;\forall u,v\in R. \mlabel{eq:aopa}
\end{equation}
\end{enumerate}
}
\mlabel{de:aop}
\end{defn}

\begin{remark}
{\rm
\begin{enumerate}
\item
Obviously, for the $A$-bimodule $\bfk$-algebra $(A,\apr,L,R)$, an $\calo$-operator $\alpha: (A,\apr,L,R)\to A$ of weight $\lambda$ is just a Rota-Baxter operator on $(A,\apr)$ of the same weight. An $\calo$-operator can be viewed as
a relative version of a Rota-Baxter operator in the sense that the
domain and range of an $\calo$-operator might be different.
\item
The construction of $\calo$-operators of $\lambda=0$ has been defined by Uchino~\mcite{Uc} under the name of a generalized Rota-Baxter operator who also obtained Theorem~\mref{thm:aopdend}.(\mref{it:aopdd}).
\end{enumerate}
}
\end{remark}

We note the following simple relationship between $\calo$-operators on modules and $\calo$-operators on algebras of weight zero.

\begin{lemma}
Let $A$ be a $\bfk$-algebra.
If $\alpha:R\to A$ is an $\calo$-operator on a $\bfk$-algebra $(R,\rpr)$ of weight zero, then $\alpha$ is an $\calo$-operator on the underlying $\bfk$-module of $(R,\rpr)$. Conversely, let $\alpha:V\to A$ be an $\calo$-operator on a $\bfk$-module $V$. Equip $V$ with an associative multiplication (say the zero multiplication) $\rpr$. Then $\alpha$ is an $\calo$-operator on the algebra $(V,\circ)$ of weight zero.
\mlabel{lem:zero}
\end{lemma}

Thus we have natural maps between $\calo$-operators on an algebra of weight zero and $\calo$-operators on a module. But the map from $\calo$-operators on a module to $\calo$-operators on an algebra of weight zero is not canonical in the sense that it depends on a choice of a multiplication on the module which will play a subtle role later in the paper (See the remark before Theorem~\mref{thm:ontov}). Thus we would like to distinguish these two kinds of $\calo$-operators.

\subsection{Rota-Baxter algebras and dendriform algebras}
\mlabel{ss:rbd}
Generalizing the concept of a dendriform dialgebra of Loday defined in Section~\mref{sec:int}, the concept of a dendriform trialgebra was introduced by Loday and Ronco~\mcite{LR}.

\begin{defn}
{\rm {\bf (\mcite{LR})} Let $\bfk$ be a commutative ring.
A {\bf dendriform $\bfk$-trialgebra} is a quadruple $(T,\prec,\succ,\spr)$ consisting of a $\bfk$-module $T$ and three bilinear products $\prec$, $\succ$ and $\spr$ such that
\allowdisplaybreaks{
\begin{eqnarray}
&&(x\prec y)\prec z=x\prec (y\star z),\
(x\succ y)\prec z=x\succ (y\prec z),\notag \\
&&(x\star y)\succ z=x\succ (y\succ z),\
(x\succ y)\spr z=x\succ (y\spr z),\
\mlabel{eq:tri}\\
&&(x\prec y)\spr z=x\spr (y\succ z),\
(x\spr y)\prec z=x\spr (y\prec z),\
(x\spr y)\spr z=x\spr (y\spr z) \notag
\end{eqnarray}
}
for all $x,y,z\in T$. Here
$\star=\prec+\succ+\spr.$
}
\end{defn}
\begin{prop}
 $($\mcite{Lo,LR}$)$\quad Given a dendriform dialgebra $(D, \prec, \succ)$ (resp. dendriform trialgebra $(D,\prec,\succ,\spr)$).
The product given by
\begin{equation}
x\star y=x\prec y+x\succ y,\;\;\forall x,y\in D
\mlabel{eq:split}
\end{equation}
{\rm (}resp.
\begin{equation}
x\star y=x\prec y+x\succ y+x\spr y,\;\;\forall x,y\in D)
\mlabel{eq:split3}
\end{equation}
defines an associative algebra product on $D$.
\mlabel{pp:loday}
\end{prop}
We summarize Proposition~\mref{pp:loday} by saying that dendriform dialgebra (resp. trialgebra) gives a {\bf splitting} of the associative multiplication $\star$.

Generalizing Theorem~\mref{thm:ag}, Ebrahimi-Fard~\mcite{E1} showed that, if $(R,\rpr,P)$ is a Rota-Baxter algebra of
weight $\lambda\neq 0$, then the multiplications
\begin{equation}
x\prec_P y: = x\rpr P(y),\quad x\succ_P y: = P(x)\rpr y, \quad x\spr_P y: =\lambda
x\rpr y,\quad \forall x,y\in R, \mlabel{eq:rbdt}
\end{equation}
defines a dendriform trialgebra $(R,\prec_P,\succ_P,\spr_P)$.

For a given $\bfk$-module $V$, define
\begin{eqnarray}
\mathbf{RB}_\lambda(V):&=& \left \{(V,\rpr,P)\ \Big|\ \begin{array}{l} (V,\rpr) \text{ is an }\bfk- \text{algebra } \text{ and } \\
P \text{ is a Rota-Baxter operator of weight } \lambda \text{ on } (V,\rpr)\end{array} \right\},
\mlabel{eq:rbset}\\
\mathbf{DD}(V):&=&\{ (V,\prec,\succ)\ |\ (V,\prec,\succ) \text{ is a dendriform dialgebra} \},
\mlabel{eq:ddset}
\\
\mathbf{DT}(V):&=&\{ (V,\prec,\succ,\spr)\ |\ (V,\prec,\succ,\spr) \text{ is a dendriform trialgebra} \}.
\mlabel{eq:dtset}
\end{eqnarray}
Then Eq.~(\mref{eq:rbdt}) yields a map
\begin{equation}
\Phi_{V,\lambda}: \xymatrix{\mathbf{RB}_\lambda(V) \ar[rrr] &&& \mathbf{DT}(V)
}
\mlabel{eq:relnt}
\end{equation}
which, when $\lambda=0$, reduces to the map
\begin{equation}
\Phi_{V,0}: \xymatrix{\mathbf{RB}_0(V) \ar[rrr] &&& \mathbf{DD}(V)
}
\mlabel{eq:relnd}
\end{equation}
{}from Theorem~\mref{thm:ag}.
Thus deriving all dendriform dialgebras (resp. trialgebras) on $V$ from Rota-Baxter operators on $V$ amounts to the surjectivity of $\Phi_{V,0}$ (resp. $\Phi_{V,\lambda}$).

Unfortunately this map is quite far away from being surjective. As an example, consider the rank two free $\bfk$-module $V:=\bfk e_1\oplus \bfk e_2$ with $\bfk=\CC$. In this case, $\mathbf{RB}_0(V)$, namely the set of Rota-Baxter operators of weight zero that could be defined on $V$, was computed in~\mcite{LHB}. Then through the map $\Phi_{V,0}$ above, these Rota-Baxter operators give the following six dendriform dialgebras on $V$ (products not listed are taken to be zero):
\begin{eqnarray*}
& \hspace{-2.5cm}(1).\; e_i\succ e_j=e_i\prec e_j=0;
 &(2).\; e_2\succ e_2=e_2\prec e_2=\frac{1}{2}e_1; \\
&  (3).\; e_1\succ e_1=e_1, e_1\succ e_2=e_2\prec e_1=e_2;
& (4).\; e_2\prec e_2=e_1; \\
& (5).\;  e_1\prec e_1=e_1, e_1\succ e_2=e_2\prec e_1=e_2;
& (6).\; e_2\succ e_2=e_1.
\end{eqnarray*}
However, according to~\mcite{Z}, there are at least the following additional five dendriform dialgebras on $V$ (products not listed are taken to be zero):
\begin{eqnarray*}
& \hspace{-2.5cm} (1).\; e_1\prec e_1=e_1, e_2\succ e_2=e_2;
& (2).\; e_2\succ e_1=e_2, e_1\prec e_1=e_1, e_1\prec e_2=e_2; \\
& (3).\; e_1\prec e_2=-e_2, e_1\succ e_1=e_1, e_1\succ e_2=e_2;
& (4).\; e_1\prec e_1=e_2, e_1\prec e_1=-e_2; \\
& \hspace{-2.2cm} (5).\; e_1\prec e_1=\frac{1}{3}e_2, e_1\succ e_1=\frac{2}{3}e_2. &
\end{eqnarray*}
Thus we could not expect to recover all dendriform dialgebras and trialgebras from Rota-Baxter operators. We will see that this situation will change upon replacing Rota-Baxter operators by $\calo$-operators.

\subsection{From $\calo$-operators to dendriform algebras on the domains}
\mlabel{ss:od1}
We first show that the procedure of deriving dendriform dialgebras and trialgebras from Rota-Baxter operators can be generalized to $\calo$-operators.
\begin{theorem}
Let $(A,\apr)$ be an associative algebra.
\begin{enumerate}
\item
Let $(R,\twopr,\ell,r)$ be an $A$-bimodule $\bfk$-algebra. Let $\aop:R\rightarrow A$ be an $\calo$-operator on the algebra $R$ of weight $\lambda$. Then the multiplications
\begin{equation}
u\prec_\aop v:=ur(\aop(v)),\;\;\
u\succ_\aop v:=\ell(\aop(u))v,\;\;
u\spr_\aop v: = \lambda\, u\twopr v,\;\; \forall u,v\in R,
\mlabel{eq:aopdend}
\end{equation}
define a dendriform
trialgebra $(R,\prec_\aop,\succ_\aop,\spr_\aop)$.
Further, the multiplication $\star_\aop:=\prec_\aop + \succ_\aop +\, \spr_\aop$ on $R$ defines an associative product on $R$ and the map $\aop: (R,\star_\aop) \to (A,\apr)$ is a $\bfk$-algebra homomorphism.
\mlabel{it:triax}
\item
Let $(V,\ell,r)$ be an $A$-bimodule. Let $\aop: V\to A$ be an $\calo$-operator on the module $V$.
Then the multiplications
\begin{equation}
u\prec_\aop v:=ur(\aop(v)),\;\;\
u\succ_\aop v:=\ell(\aop(u))v,\;\; \forall u,v\in V,
\mlabel{eq:aopdd}
\end{equation}
define a dendriform dialgebra $(V,\prec_\aop,\succ_\aop)$. Further,
the multiplication $\star_\aop:=\prec_\aop + \succ_\aop$ on $V$ defines an associative product and $\aop: (V,\star_\aop) \to (A,\apr)$ is a $\bfk$-algebra homomorphism.
\mlabel{it:aopdd}
\end{enumerate}
\mlabel{thm:aopdend}
\end{theorem}

\begin{proof}
(\mref{it:triax}) For any $u,v,w\in R$, by the definitions of
$\prec_\alpha, \succ_\alpha, \spr_\alpha$ and $A$-bimodule $\bfk$-algebras, we have
\begin{eqnarray*}
(u\prec_\aop v)\prec_\aop w &=& (u\prec_\aop v)\,r (\aop(w))=(u\,r(\aop(v)))\,r(\aop(w)) \quad \text{(by Eq.~(\mref{eq:aopdend}))} \\
&=&
u\,r(\aop(v)\aop(w)) \quad \text{(by Eqs.~(\mref{eq:twoalg1})-(\mref{eq:twoalg3}))}\\
&=& u\,r\big(\aop(\ell(\aop(v))w)+\aop(vr(\aop(w)))+\lambda \aop(v\twopr w)\big) \quad \text{(by Eq.~(\mref{eq:aopa}))}\\
&=&u\prec_\aop \big(\ell(\aop(v))w+ v\,r(\aop(w))+\lambda v\twopr w\big) \quad \text{(by Eq.~(\mref{eq:aopdend}))}\\
&=&u\prec_\aop (v\succ_\aop w)+u\prec_\aop (v\prec_\aop w)+u\prec_\aop(v\,\spr_\aop\, w)
\quad \text{(by Eq.~(\mref{eq:aopdend}))}.
\end{eqnarray*}
Similar arguments can be applied to verify the other axioms for a dendriform trialgebra in Eq.~(\mref{eq:tri}).
\smallskip

The second statement follows from Proposition~\mref{pp:loday} and the definition of $\aop$:
$$ \aop(u\star_\aop v)=\aop(u\prec_\aop v + u\succ_\aop v + u\,\spr_\aop\, v)
= \aop( ur(\aop(v)) + \ell(\aop(u))v + \lambda\, u\twopr v)
= \aop(u)\apr \aop(v).$$

\noindent
(\mref{it:aopdd}) By Lemma~\mref{lem:zero}, when we equip $V$ with the zero multiplication $\rpr$, the $\calo$-operator $\alpha:V\to A$ on the module becomes an $\calo$-operator on the algebra $(R,\rpr)$ of weight zero. Then by Item~(\mref{it:triax}), $(V,\prec_\alpha,\succ_\alpha,\spr_\alpha)$ is a dendriform trialgebra which is in fact a dendriform dialgebra since $\spr_\alpha$ is zero.
\end{proof}

For a $\bfk$-algebra $A$ and an $A$-bimodule $\bfk$-algebra $(R,\rpr)$, denote
\begin{equation}
\oaset_\lambda(R,A):=\oaset_{\lambda}((R,\rpr),A):= \{\alpha:R\to A\ |\ \alpha \text{ is an } \calo\text{-operator on the algebra } R \text{ of weight } \lambda\}.
\mlabel{eq:opseta}
\end{equation}
By Theorem~\mref{thm:aopdend}.(\mref{it:triax}), we obtain a map
\begin{equation}
\oamap_{\lambda,R,A}: \oaset_\lambda((R,\rpr),A) \longrightarrow \mathbf{DT}(|R|),
\mlabel{eq:odt}
\end{equation}
where $|R|$ denotes the
underlying $\bfk$-module of $R$.

Now let $V$ be a $\bfk$-module.
Let $\oaset_\lambda(V,-)$ denote the set of $\calo$-operators on the algebra $(V,\rpr)$ of weight $\lambda$, where $\rpr$ is an associative product on $V$. In other words,
$$\oaset_\lambda(V,-):=\coprod_{R,A} \oaset_\lambda(R,A),$$
where the disjoint union runs through all pairs $(R,A)$ where $A$ is a $\bfk$-algebra and $R$ is an $A$-bimodule $\bfk$-algebra such that $|R|= V$. Then from the map $\oamap_{\lambda,V,A}$ in Eq.~(\mref{eq:odt}) we obtain
\begin{equation}
\oamap_{\lambda,V}:=\coprod_{R,A} \oamap_{\lambda,V,A}: \oaset_\lambda(V,-) \longrightarrow \mathbf{DT}(V).
\mlabel{eq:odt2}
\end{equation}

Similarly, for a $\bfk$-module $V$ and $\bfk$-algebra $A$, denote
\begin{equation}
\omset(V,A)=\{\alpha:V\to A\ |\ \alpha \text{ is an } \calo\text{-operator on the module }V\}.
\mlabel{eq:opsetv}
\end{equation}
By Theorem~\mref{thm:aopdend}.(\mref{it:aopdd}), we obtain a map
\begin{equation}
\ommap_{V,A}: \omset(V,A) \longrightarrow \mathbf{DD}(V).
\mlabel{eq:odd}
\end{equation}
Let $\omset(V,-)$ denote the set of $\calo$-operators on the module $V$. In other words,
$$\omset(V,-):=\coprod_{A} \omset(V,A),$$
where $A$ runs through all the $\bfk$-algebras. Then we have
\begin{equation}
\ommap_{V}:=\coprod_{A} \ommap_{V,A}: \omset(V,-) \longrightarrow \mathbf{DD}(V).
\mlabel{eq:odd2}
\end{equation}

Let us compare $\oamap_{0,V}$ and $\ommap_{V}$ for a $\bfk$-module $V$. For a given associative multiplication $\rpr$ on $V$, we have the natural bijection $\oaset_{0}((V,\rpr),-) \to \omset(V,-)$ sending an $\calo$-operator $\alpha:(V,\rpr)\to A$ on the algebra $(V,\rpr)$ to the $\calo$-operator $\alpha:V\to A$ on the underlying $\bfk$-module $V$. Thus $\oaset_{0}(V,-)$ is the disjoint union of multiple copies of $\omset(V,-)$, one copy for each associative multiplication on $V$. Therefore, the surjectivity of $\ommap_{V}$ is a stronger property than the surjectivity of $\oamap_{0,V}$.

\begin{theorem} Let $V$ be a $\bfk$-module. The maps
$\oamap_{1,V}$ and $\ommap_V$ are surjective.
\mlabel{thm:ontov}
\end{theorem}
By this theorem, all dendriform dialgebra (resp. trialgebra) structures on $V$ could be recovered from $\calo$-operators on the module (resp. on the algebra).
\begin{proof}
We first prove the surjectivity of $\oamap_{1,V}$. Let $(V, \prec,\succ, \spr)$
be a dendriform trialgebra.
By Proposition~\mref{pp:loday}, $V$ becomes a $\bfk$-algebra with the product $\apr:=\prec+\succ+\,\spr$.
Define two linear maps
\begin{equation}
 L_\succ, R_\prec: V\rightarrow \End_\bfk(V), \quad L_\succ(x)(y)=x\succ y, \;\; R_\prec (x)(y)=y\prec x, \;\;x,y\in V.
\mlabel{eq:denlr}
\end{equation}
Then it is straightforward to check that the dendriform trialgebra axioms of $(V,\prec,\succ,\,\spr\,)$ imply that $(V,\spr,L_\succ, R_\prec)$ satisfies all the axioms in Eq.~(\mref{eq:twoalg1}) -- (\mref{eq:twoalg3}) for a $(V,\apr)$-bimodule $\bfk$-algebra. For example,
$$L_\succ(x\apr y) z = (x\prec y + x\succ y + x \spr y) \succ z
= x\succ (y\succ z)=L_\succ(x) (L_\succ (y)(z)), \quad \forall x,y,z\in V.$$
Also the identity linear map
$$\id: (V, \spr, L_{\succ},R_{\prec}) \rightarrow (V,\apr)$$
from the $(V,\apr)$-bimodule $\bfk$-algebra $(V,\spr\,,L_\succ,R_\prec)$ to the $\bfk$-algebra $(V,\apr)$ is an $\calo$-operator on the algebra $(V,\spr)$ of weight 1:
\begin{equation}
 \id(x) \apr \id(y)=x\apr y = x \prec y + x \succ y + x\spr y
= \id (x R_\prec (\id(y))) + \id (L_\succ (\id(x)) y)+ \id
(x\spr y),
\mlabel{eq:ido}
\end{equation}
$\forall x,y\in V.$
Further, by Eq.~(\mref{eq:aopdend}), we have
$\prec_{\id}\; =\; \prec, \succ_{\id}\;=\;\succ$ and $\spr_{\id}\;=\;\spr$.
Thus $(V,\prec,\succ,\spr)$ is the image of the $\calo$-operator $\id: (V, L_{\succ},R_{\prec},\spr) \rightarrow (V,\apr)$ under the map $\oamap_{1,V}$, showing that $\oamap_{1,V}$ is surjective.
\smallskip

To prove the surjectivity of $\ommap_V$, let $(V,\prec,\succ)$ be a dendriform dialgebra. Then by equipping $V$ with the zero multiplication $\spr=0$, we obtain a dendriform trialgebra $(V,\prec,\succ,\spr)$. Let $\apr=\prec+\succ+\,\spr$. Then by the proof of the surjectivity of $\oamap_{0,V}$ we have the $(V,\apr)$-bimodule $\bfk$-algebra $(V,\spr,L_\succ,R_\prec)$ defined by Eq.~(\mref{eq:denlr}) and the $\calo$-operator
$\id: (V,\spr,L_\succ,R_\prec) \to (V,\apr)$ on the algebra of weight 1 such that $\oamap_{1,V}(\id)=(V,\prec,\succ,\spr)$.
Since $\spr=0$, we see that Eq.~(\mref{eq:ido}) satisfied by $\id$ as an $\calo$-operator on the algebra $(V,\spr)$ is also the equation for the map $\id$ to be an $\calo$-operator on the module $V$. Further $\ommap_V(\id)=\oamap_{0,V}(\id)=(V,\prec,\succ)$. This proves the surjectivity of $\ommap_V$.
\end{proof}

\section{$\calo$-operators and dendriform algebras on the ranges}
\mlabel{sec:oda}
We next study another kind of relationship between $\calo$-operators and dendriform dialgebras and trialgebras by focusing on the algebra $(A,\apr)$ in an $\calo$-operator $\alpha: R\to A$. We first show that, under a natural condition, an $\calo$-operator $\alpha: R\to A$ on the module (resp. on the algebra) gives a dendriform dialgebra (resp. trialgebra) structure on $A$ that gives a splitting of $\apr$ in the sense of Proposition~\mref{pp:loday} (see the remark thereafter).
We then show that the $\calo$-operators $\alpha: R\to A$, as the $\bfk$-module (resp. $\bfk$-algebra) $R$ varies, recover all dendriform dialgebra or trialgebra structures on $(A,\apr)$ with the splitting property. We in fact give bijections between suitable equivalence classes of these $\calo$-operators and (equivalent classes of) dendriform dialgebras and trialgebras.

\subsection{From $\calo$-operators to dendriform algebras on the ranges}
\mlabel{ss:od2}

We first give the following consequence of Theorem~\mref{thm:aopdend}, providing a dendriform dialgebra or a trialgebra on the range of an $\calo$-operator.
\begin{prop}
Let $(A,\apr)$ be a $\bfk$-algebra.
\begin{enumerate}
\item
Let $(R,\twopr,\ell,r)$ be an $A$-bimodule $\bfk$-algebra. Let $\aop:R\rightarrow A$ be an
$\calo$-operator on the algebra of weight $\lambda$.
If $\ker \aop$ is an ideal of $(R,\twopr)$, then there is a dendriform trialgebra structure on $\aop(R)$ given by
\begin{eqnarray}
&&\aop(u)\prec_{\aop,A} \aop(v):=\aop(ur(\alpha(v))),\quad
\aop(u)\succ_{\aop,A} \aop(v):=\aop(\ell(\alpha(u))v), \notag \\
&&\aop(u)\,\spr_{\aop,A}\, \aop(v):=\aop(\lambda u\twopr v),\quad \forall u,v\in R.
\mlabel{eq:tdend}
\end{eqnarray}
Furthermore, $\apr=\prec_{\aop,A} + \succ_{\aop,A} + \spr_{\aop,A}$ on $\aop(R)$.
In particular, if the $\calo$-operator $\aop$ is invertible (that is, bijective as a $\bfk$-linear map), then the multiplications
\begin{eqnarray}
&& x\prec_{\aop,A} y:=\aop(\aop^{-1}(x) r(y)),\quad
x\succ_{\aop,A} y:=\aop(\ell(x)\aop^{-1}(y)),\notag \\
&& x\,\spr_{\aop,A}\, y:=\aop(\lambda \aop^{-1}(x)\twopr \aop^{-1}(y)),\quad \forall x,y\in A,
\mlabel{eq:tdend2}
\end{eqnarray}
define a dendriform trialgebra $(A,\prec_{\aop,A},\succ_{\aop,A},\spr_{\aop,A})$
such that $\apr=\prec_{\aop,A} + \succ_{\aop,A} + \spr_{\aop,A}$ on $A$, called the {\bf dendriform trialgebra on the range} of $\alpha$.
\mlabel{it:indend}
\item
Let $(V,\ell,r)$ be a $A$-bimodule. Let $\aop:V\rightarrow A$ be an
invertible $\calo$-operator on the module.
Then
\begin{equation}
x\prec_{\aop,A} y:=\aop(\aop^{-1}(x) r(y)),\;\;
x\succ_{\aop,A} y: =\aop(\ell(x)\aop^{-1}(y)),\;\;
\forall x,y\in A,
\mlabel{eq:ddend2}
\end{equation}
define a dendriform dialgebra $(A,\prec_{\aop,A},\succ_{\aop,A})$ on $A$
such that $\apr=\prec_{\aop,A} + \succ_{\aop,A}$ on $A$, called the {\bf dendriform dialgebra on the range} of $\alpha$.
\mlabel{it:indendd}
\end{enumerate}
\mlabel{pp:indend}
\end{prop}

\begin{proof}
\noindent
(\mref{it:indend})
We first prove that the multiplications in Eq.~(\mref{eq:tdend}) are well-defined. More precisely, for $u,u',v,v'\in R$ such that
$\aop(u)=\aop(u')$ and $\aop(v)=\aop(v')$, we check that
\begin{equation}
\aop(ur(\aop(v)))=\aop(u'r(\aop(v'))),\
\aop(\ell(\aop(u))v)=\aop(\ell(\aop(u'))v'),\ \aop(u\twopr v)=\aop(u'\twopr v').
\mlabel{eq:awell}
\end{equation}
But since $u-u'$ and $v-v'$ are in $\ker \alpha$, we have
$$0=\aop(u-u')\aop(v)=\aop(\ell (\aop(u-u'))v)+\aop((u-u')r(\aop(v)))+\lambda \aop((u-u')\twopr v)$$
with the first term on the right hand side vanishing. The third term also vanishes since $\ker \aop$ is an ideal of $(R,\twopr)$. Thus the second term also vanishes and $(u-u')r(\aop(v))$ is in $\ker \aop$.
We then find that
$$
ur(\aop(v))-u'r(\aop(v'))=
(u-u')r(\aop(v)) + u'r(\aop(v-v'))$$
is in $\ker \aop$. This verifies the first equation in Eq.~(\mref{eq:awell}). The other two equations are verified similarly.
Then the axioms in Eq.~(\mref{eq:tri}) for $(\aop(R),\prec_{\aop,A},\succ_{\aop,A},\,\spr_{\aop,A}\,)$ to be a dendriform trialgebra follows from the axioms for $(R,\prec_\aop,\succ_\aop,\,\spr_\aop\,)$ to be a dendriform trialgebra.

Since $\alpha$ is an $\calo$-operator, we have
$$ \alpha(u)\apr \alpha(v)
=\aop(ur(\aop(v)))+\aop(\ell(\aop(u))v)+\aop(u\twopr v)
= u \prec_{\aop,A} v + u \succ_{\aop,A} v + u \spr_{\aop,A} v, \quad \forall u,v \in R.$$
This proves the second statement in Item~(\mref{it:indend}).
Then the last statement follows as a direct consequence.
\smallskip

\noindent
(\mref{it:indendd}). Let $(V,\ell,r)$ be an $A$-bimodule and let $\alpha:V\to A$ be an $\calo$-operator on the module $V$. Then when $V$ is equipped with an associative multiplication $\rpr$ (say $\rpr\equiv 0$), $\alpha$ becomes an $\calo$-operator on the algebra $(V,\rpr)$ of weight zero. Then by Item~(\mref{it:indend}), $(V,\rpr,\prec_{\alpha,A},\succ_{\alpha,A}, \spr_{\alpha,A})$ is a dendriform trialgebra such that $\apr=\prec_{\alpha,A}+\succ_{\alpha,A}+ \spr_{\alpha,A}$. But since $x \spr_{\alpha,A}y =\alpha(0 \alpha^{-1}(x) \rpr \alpha^{-1}(y))=0, \forall x,y\in A,$ we see that $(V,\prec_{\alpha,A},\succ_{\alpha,A})$ is a dendriform dialgebra such that $\apr=\prec_{\alpha,A}+\succ_{\alpha,A}.$

\end{proof}

Proposition~\mref{pp:indend} motivate us to introduce the following notations.
\begin{defn}
Let  $(A,\apr)$ be a $\bfk$-algebra.
\begin{enumerate}
\item
Let $\ioaset_\lambda(A,\apr)$ (resp. $\iomset(A,\apr)$) denote the set of invertible (i.e., bijective) $\calo$-operators $\alpha:R\to A$ on the algebra of weight $\lambda\in\bfk$ (resp. on the module), where $R=(R,\circ,\ell,r)$ is an $A$-bimodule $\bfk$-algebra (resp. $R=(R,\ell,r)$ is an $A$-module).
\mlabel{it:inop}
\item
Let ${\bf DT}(A,\apr)$ $($resp. ${\bf DD}(A,\apr)$$)$ denote the set of dendriform
trialgebra $($resp.  dialgebra$)$ structures $(A,\prec,\succ,\spr)$ $($resp. $(A,\prec,\succ)$$)$ on $(A,\apr)$ such that $\apr=\prec+\succ+\;\spr$ $($resp. $\apr=\prec+\succ$$)$.
\mlabel{it:resdt}
\item
Let
\begin{equation}
\ioamap_A: \ioaset(A,\apr) \longrightarrow {\bf DT}(A,\apr), \quad
\aop \mapsto (\prec_{\aop,A},\succ_{\aop,A},\spr_{\aop,A})
\mlabel{eq:opdta}
\end{equation}
\begin{equation}
(resp.\quad \iommap_A: \iomset(A,\apr) \longrightarrow {\bf
DD}(A,\apr), \quad \aop \mapsto (\prec_{\aop,A},\succ_{\aop,A}).)
\mlabel{eq:opdda}
\end{equation}
be the maps defined by Proposition~\mref{pp:indend}.
\end{enumerate}
\mlabel{de:inop}
\end{defn}

\subsection{Bijective correspondences}
\mlabel{ss:od3}

Instead of proving just the surjectivities of the maps $\ioamap_A$ and $\iommap_A$ defined by Eq.~(\mref{eq:opdta}) and Eq.~(\mref{eq:opdda}), we give a more quantitative description of these maps.

We first prove a lemma that justifies the concepts that will be introduced next.
\begin{lemma}
Let $(A,\apr)$ be a $\bfk$-algebra and let $(R,\circ,\ell,r)$ be an $A$-bimodule $\bfk$-algebra. Let $\aop:(R,\rpr,\ell,r) \to A$ be an $\calo$-operator on the algebra $(R,\rpr)$ of weight $\lambda$.
\begin{enumerate}
\item
Let $g:(R_1,\circ_1,\ell_1,r_1)\to (R,\circ,\ell,r)$ be an isomorphism of $A$-bimodule $\bfk$-algebras. Then $\alpha  g: (R_1,\circ_1,\ell_1,r_1) \to A$ is an $\calo$-operator on the algebra $(R_1,\rpr_1)$ of weight $\lambda$.
\mlabel{it:autor}
\item
Let $f:A\to A$ be a $\bfk$-algebra automorphism. Then $f \aop: (R,\circ, \ell f^{-1}, r  f^{-1}) \to A$ is an $\calo$-operator on the algebra $(R,\rpr)$ of weight $\lambda$.
\mlabel{it:autoa}
\end{enumerate}
Similar statements hold for an $A$-bimodule $V$ in place of an $A$-bimodule $\bfk$-algebra $R$.
\mlabel{lem:oop}
\end{lemma}
\begin{proof}
(\mref{it:autor}) For all $x,y\in R_1$, we have
\begin{eqnarray*}
(\alpha\circ g)(x)\apr (\alpha\circ g)(y)&=&
\alpha(g(x))\apr \alpha(g(y)) \\
&=& \alpha((\ell\alpha(g(x)))g(y))+\alpha(g(x)(r\alpha)(g(y))) +\lambda\alpha(g(x) \circ g(y))\\
&=& \alpha[g(\ell_1(\alpha(g(x)))y)] + \alpha[x r_1(\alpha(g(y)))] + \lambda \alpha[g(x \circ_1 y)]\\
&=& (\alpha \circ g)(\ell_1((\alpha\circ g)(x))y) + (\alpha\circ g)(x r_1((\alpha\circ g)(y))) +\lambda (\alpha\circ g)(x\circ_1 y).
\end{eqnarray*}
Thus $\alpha\circ g$ is an $\calo$-operator of weight $\lambda$. \smallskip

\noindent
(\mref{it:autoa}) Let $f:(A,\apr)\to (A,\apr)$ be a $\bfk$-algebra automorphism. It is easy to verify that $(R,\rpr,\ell f^{-1},r f^{-1})$ satisfies all the axioms of an $A$-bimodule $\bfk$-algebra. For example, the first equation in Eq.~(\mref{eq:twoalg1}) holds since
$$ (\ell f^{-1})(x\apr y) v = \ell(f^{-1}(x)\apr f^{-1}(y)) v = \ell(f^{-1}(x))(\ell(f^{-1}(y))v) = (\ell f^{-1})((\ell f^{-1})(y))v.$$
Further, $f\alpha: (R,\rpr,\ell f^{-1},r f^{-1})\to (A,\apr)$ is an $\calo$-operator since
\begin{eqnarray*}
(f\alpha)(x) \apr (f\alpha)(y) &=& f(\alpha(x) \apr \alpha(y)) \\
&=&f\big( \alpha(\ell(\alpha(x)) y) +\alpha(x \ell(r(y)))+\lambda \alpha( x \rpr y)\big) \\
&=& (f\alpha)\big((\ell f^{-1})((f\alpha)(x)y\big) + (f\alpha)\big(x(r f^{-1})((f\alpha)(y))\big) + \lambda (f\alpha)(x \rpr y).
\end{eqnarray*}

The proofs of the statements for $\calo$-operators on an $A$-bimodule $V$ in place of an $A$-bimodule $\bfk$-algebra are obtained by equipping $V$ with the zero multiplication and following the same argument as Theorem~\mref{thm:ontov}.
\end{proof}

We can now define equivalence relations among $\calo$-operators and dendriform algebras.
\begin{defn} {\rm
Let  $(A,\ast)$ be a $\bfk$-algebra.
\begin{enumerate}
\item
For $A$-bimodule $\bfk$-algebras $(R_i,\circ_i,\ell_i,r_i)$ and invertible $\calo$-operators
$\aop_i:R_i\to A, i=1,2,$ call $\aop_1$ and $\aop_2$ {\bf
isomorphic}, denoted by $\aop_1\cong \aop_2$, if there is an
isomorphism $g:(R_1,\circ_1,\ell_1,r_1) \to (R_2,\circ_2,\ell_2,r_2)$ of $A$-bimodule $\bfk$-algebras (see Definition~\mref{de:bimal}) such that $\aop_1=\aop_2 g$. Similarly define isomorphic invertible $\calo$-operators on modules.
\item
For $A$-bimodule $\bfk$-algebras $(R_i,\circ_i,\ell_i,r_i)$ and invertible $\calo$-operators
$\aop_i:R_i\to A, i=1,2,$ call $\aop_1$ and $\aop_2$ {\bf
equivalent}, denoted by $\aop_1\sim \aop_2$, if there exists a
$\bfk$-algebra automorphism $f:A\to A$ such that
$f \alpha_1 \cong \alpha_2$. In other words, if there exist a $\bfk$-algebra automorphism $f:A\to A$ and an isomorphism $g:(R_1,\circ_1,\ell_1 f^{-1},r_1 f^{-1})\to (R_2,\circ_2,\ell_2 ,r_2 )$ of $A$-bimodule $\bfk$-algebras such that $f \alpha_1= \alpha_2 g$. Similar define equivalent invertible $\calo$-operators on modules.
\item
Let $\ioaset(A,\apr)/{\cong}$ $($resp. $\ioaset(A,\apr)/{\sim}$$)$ denote the set of
equivalent classes from the relation $\cong$ $($resp. $\sim$$)$. Similarly define $\iomset(A,\apr)/{\cong}$ and $\iomset(A,\apr)/\sim$.
\item
Two dendriform trialgebras $(A,\prec_i,\succ_i,\cdot_i), i=1,2,$ on
$A$ are called {\bf isomorphic}, denoted by
$(A,\prec_1,\succ_1,\cdot_1)\cong (A,\prec_2,\succ_2,\cdot_2)$ if
there is a linear bijection $F:A\to A$ such that
$$F(x\prec_1 y)=F(x)\prec_2 F(y), \quad F(x\succ_1 y)=F(x)\succ_2 F(y), \quad F(x\cdot_1y)=F(x)\cdot_2F(y), \quad \forall x,y\in A.$$
\item
Two dendriform
dialgebras $(A,\prec_i,\succ_i), i=1,2,$ on $A$ are called
{\bf isomorphic}, denoted by $(A,\prec_1,\succ_1)\cong
(A,\prec_2,\succ_2)$ if there is a linear bijection $F:A\to A$ such
that
$$F(x\prec_1 y)=F(x)\prec_2 F(y), \quad F(x\succ_1 y)=F(x)\succ_2
F(y), \quad \forall x,y\in A.$$
\item
Let ${\bf DT}(A,\apr)/\cong$ $($resp. ${\bf DD}(A,\apr)/\cong$$)$ denote the set of
equivalent classes of ${\bf DT}(A,\apr)$ $($resp. ${\bf DD}(A,\apr)$$)$ modulo the isomorphisms.
\end{enumerate}
}
\mlabel{de:iso}
\end{defn}

\begin{theorem} Let $(A,\apr)$ be a $\bfk$-algebra. Let
$$\ioamap_A: \ioaset(A,\apr) \la {\bf DT}(A,\apr), \quad \alpha\mapsto
(A,\prec_{\alpha,A},\succ_{\alpha,A},\cdot_{\alpha,A}),
$$
be the map defined by Eq.~(\mref{eq:opdta}). Then $\ioamap_A$ induces
bijections
\begin{eqnarray}
\ioamap_{A,\cong}:&& \ioaset(A,\apr)/{\cong}\longrightarrow {\bf DT}(A,\apr),
\mlabel{eq:cong}
\\
\ioamap_{A,\sim}: && \ioaset(A,\apr)/{\sim}\; \la {\bf DT}(A,\apr)/{\cong}.
\mlabel{eq:sim}
\end{eqnarray}
In particular, $\ioamap_A$ is surjective.
\mlabel{thm:iso}
\end{theorem}
Similar statements hold for $\iommap_A$.

\begin{proof} Let $\alpha_i:(R_i,\circ_i,\ell_i,r_i)\to (A,\apr), i=1,2,$ be two isomorphic invertible $\aop$-operators.
Then there exists an isomorphism $g:(R_1,\circ_1,\ell_1,r_1)\to (R_2,\circ_2,\ell_2,r_2)$ of $A$-bimodule $\bfk$-algebras such that $\aop_1=\aop_2 g$. We see that
their corresponding dendriform trialgebras $$\ioamap_A(\alpha_1)= (A,\prec_{\alpha_1,A},\succ_{\alpha_1,A},\spr_{\alpha_1,A})
\quad \text{and}\quad
\ioamap_A(\alpha_2)= (A,\prec_{\alpha_2,A},\succ_{\alpha_2,A},\spr_{\alpha_2,A})$$ from Eq.~(\mref{eq:tdend2}) coincide since, for any $x,y\in A$, we have
$$x\prec_{\aop_1,A} y=\aop_1(\aop^{-1}_1(x)r_1(y)) =(\aop_2g)[(g^{-1}\aop_2^{-1}(x))(gr_2(y)g^{-1})]=
\aop_2(\aop_2^{-1}(x)r_2(y))=x\prec_{\aop_2,A} y,$$
$$x\succ_{\aop_1,A}
y=\aop_1(\ell_1(x)\aop_1^{-1}(y)) =(\aop_2g)[(g^{-1}\ell_2(x)g)(g^{-1}\aop^{-1}_2)(y)]=
\aop_2(\ell_2(x)\aop^{-1}_2(y))=x\succ_{\aop_2,A} y,$$
$$x\cdot_{\aop_1,A}
y=\lambda\alpha_1(\alpha_1^{-1}(x)\circ_1\alpha_1^{-1}(y))
=\lambda\alpha_2g(g^{-1}\alpha_2^{-1}(x)\circ_1g^{-1}\alpha_2^{-1}(y))
=\lambda\alpha_2(\alpha_2^{-1}(x)\circ_2\alpha_2^{-1}(y)) =x\cdot_{\aop_2,A} y.$$
Therefore the map $\ioamap_A$ induces a map $\ioamap_{A,\cong}$ on the set $\ioaset(A,\apr)/{\cong}$ of isomorphism classes of invertible $\calo$-operators on $(A,\apr)$.

Let $(A,\prec,\succ,\cdot)$ be a dendriform trialgebra. The
proof of Theorem~\mref{thm:ontov} shows that Eq.~(\mref{eq:denlr}) defines an $A$-bimodule $\bfk$-algebra
$(A, L_\succ, R_\prec,\spr)$ and an $\calo$-operator
$\alpha:=\id: (A,L_\succ, R_\prec,\cdot) \to (A,\apr)$
which is the identity on the underlying $\bfk$-module and hence is invertible. Since this $\alpha$ gives $\ioamap_A (\alpha)=(A,\succ, \prec,\cdot)$, we have proved that $\ioamap_A$, and hence $\ioamap_{A,\cong}$, is surjective.
Furthermore, let $\alpha_i:(R_i,\circ_i,\ell_i,r_i)\to (A,\apr)$ be two invertible $\calo$-operators such that
$\ioamap_A(\aop_1)=\ioamap_A(\aop_2)$. That is,
$$(A,\prec_{\alpha_1,A},\succ_{\alpha_1,A},\,\spr_{\alpha_1,A}) =(A,\prec_{\alpha_1,A},\succ_{\alpha_1,A},\,\spr_{\alpha_1,A}).$$
Define $g=\alpha_2^{-1}\alpha_1:R_1\to R_2$. {}For $x,y\in A$, from
$x\prec_{\alpha_1,A} y = x\prec_{\alpha_2,A} y$ we obtain
$$\alpha_1(\alpha_1^{-1}(x)r_1(y))
=\alpha_2(\alpha_2^{-1}(x)r_2(y)).$$ Then
$(\alpha_2^{-1}\alpha_1)(\alpha_1^{-1}(x)r_1(y))=\alpha_2^{-1}(x)r_2(y)$.
Thus for any $u_1\in R_1$, taking $x=\alpha_1(u)$, we have
$(\alpha_2^{-1}\alpha_1)(u_1r_1(y))
=(\alpha_2^{-1}\alpha_1)(u_1)r_2(y)$. By the same argument,
$(\alpha_2^{-1}\alpha_1)(\ell_1(x) v_1)=\ell_2(x)
(\alpha_2^{-1}\alpha_1)(v_1)$ for $x\in A, v_1\in R_1$. Thus
$\alpha_2^{-1}\alpha_1$ is an $A$-bimodule
homomorphism from $(R_1,\ell_1,r_1)$ to $(R_2,\ell_2,r_2)$.
Similarly, from $\spr_{\alpha_1,A}=\spr_{\alpha_2,A}$ we find that
$\alpha_2^{-1}\alpha_1$ is a $\bfk$-algebra homomorphism from
$(R_1,\rpr_1)$ to $(R_2,\rpr_2)$. Since $\alpha_2^{-1}\alpha_1$ is
also a bijection, we have proved that the $\calo$-operators
$\alpha_i:(R_i,\circ_i,\ell_i,r_i)\to (A,\apr), i=1,2,$ are
isomorphic by $g=\aop_2^{-1}\aop_1:A\to A$. Hence $\Phi_\cong$ is
also injective, proving Eq.~(\mref{eq:cong}).
\medskip

We next prove Eq.~(\mref{eq:sim}). Let $\alpha_i:(R_i,\circ_i,\ell_i,r_i)\to (A,\apr), i=1,2,$ be two equivalent invertible $\aop$-operators. Then
there exist a $\bfk$-algebra automorphism $f:A\rightarrow A$ and an isomorphism $g:(R_1,\circ_1,\ell_1 f^{-1},r_1 f^{-1})\rightarrow (R_2,\circ_2,\ell_2,r_2)$ of $A$-bimodule $\bfk$-algebras such that
$f \aop_1=\aop_2 g$.
Consider the corresponding dendriform trialgebras $$\ioamap_A(\alpha_1)= (A,\prec_{\alpha_1,A},\succ_{\alpha_1,A},\spr_{\alpha_1,A})
\quad \text{and} \quad
\ioamap_A(\alpha_2)= (A,\prec_{\alpha_2,A},\succ_{\alpha_2,A},\spr_{\alpha_2,A})$$ from Eq.~(\mref{eq:tdend2}). By the definition of $A$-bimodule isomorphisms, for $x,y\in A$, we have
\begin{eqnarray*}
f(x\succ_{\alpha_1,A} y) &=& f( \alpha_1(\ell_1(f^{-1}(f(x)))\alpha_1^{-1}(y))
\\
&=& f\big((f^{-1}\alpha_2 g)(g^{-1}\ell_2(f(x))g)(g^{-1}\alpha_2^{-1}f)(y)\big)
\\ &=& \alpha_2(\ell_2(f(x)))\alpha_2^{-1}(f(y))
\\
&=& f(x) \succ_{\alpha_2,A} f(y).
\end{eqnarray*}
Similarly,
$$  f(x\prec_{\alpha_1,A} y) = f(x) \prec_{\alpha_2,A} f(y).$$
Finally,
\begin{eqnarray*}
f(x \cdot_{\alpha_1,A} y) &=& f\big(\alpha_1\big(\lambda \alpha^{-1}_1(x) \circ_1 \alpha^{-1}_1(y)\big)\big) \\
&=& \lambda f\big((f^{-1}\alpha_2 g)((g^{-1}\alpha_2^{-1} f)(x) \circ_1 (g^{-1}\alpha_2^{-1} f)(y))\big)\\
&=& \lambda \alpha_2( \alpha_2^{-1}(f(x)) \circ_2 \alpha^{-1}(f(y)))\\
&=& f(x) \cdot_{\alpha_2,A} f(y).
\end{eqnarray*}
Hence the two dendriform trialgebras $\ioamap_A(\alpha_1)$ and $\ioamap_A(\alpha_2)$ are isomorphic through $f$.

Conversely, let $F:(A,\prec_1,\succ_1,\spr_1)\to (A,\prec_2,\succ_2,\spr_2)$ be an isomorphism of two dendriform trialgebras in ${\bf DT}(A)/{\cong}$. Since $\apr=\prec_1+\succ_1+\,\spr_1=\prec_2+\succ_2+\,\spr_2$ by definition, $F$ is also a $\bfk$-algebra automorphism of $(A,\apr)$. Let $\alpha_i:(R_i,\rpr_i,\ell_i,r_i)\to (A,\apr), i=1,2,$ be invertible $\calo$-operators such that $\ioamap_A(\alpha_i)=(A,\prec_i,\succ_i,\spr_i), i=1,2.$ To prove $\alpha_1\sim \alpha_2$ we only need to show that $g:=\alpha_2^{-1}f \alpha_1$ defines an isomorphism of $A$-bimodule $\bfk$-algebras from $(R_1,\rpr_1,\ell_1 F^{-1}, r_1 F^{-1})$ to $(R_2,\rpr_2,\ell_2,r_2)$. First, for $u\in R_1$ and $y\in A$, taking $x=\alpha_1(u)\in A$, we have
\begin{eqnarray*}
g(u(r_1F^{-1})(y))&=& \alpha_2^{-1}F\alpha_1\big(\alpha^{-1}(x)(r_1F^{-1})(y)\big) \\
&=& \alpha_2^{-1}F(x\prec_{\alpha_1,A} F^{-1}(y))\\
&=& \alpha_2^{-1}(F(x)\prec_{\alpha_2,A} y)\\
&=& \alpha_2^{-1}\big(\alpha_2(\alpha_2^{-1}(F(x))r_2(y))\big) \\
&=& (\alpha_2^{-1}F)(\alpha_1(u)) r_2(y)\\
&=& g(u)r_2(y).
\end{eqnarray*}
By the same argument, we have
$$ g((\ell_1F^{-1})(x) v)=\ell_2(x) g(v), \quad \forall x\in A,v\in R$$
and
$$ g(u \rpr_1 v)=g(u) \rpr_2 g(v), \quad \forall u,v\in R.$$
Since $g$ is also bijective, we have proved that $g$ is the isomorphism of $A$-bimodule $\bfk$-algebras that we want. This completes the proof.
\end{proof}


%
%

\end{document}